\documentclass[10pt]{amsart} 

\numberwithin{equation}{section}

\newtheorem*{theorem}{Theorem}

\newtheorem*{prop}{Proposition}

\newtheorem{remark}{Remark}

\newtheorem{lemma}{Lemma}

\newcommand{\C}{{\mathbb C}}

\newcommand{\RRR}{{\mathcal R}}
\newcommand{\I}{{\mathcal I}}
\newcommand{\eps}{\varepsilon}
\newcommand{\M}{{\mathcal{M}}}
\newcommand{\U}{{\mathcal{U}}}
\def\R{{\mathbb R}}
\def\S{{\mathbb S}}
\def \N {\mathbb{N}}

\def\H{{\mathcal H}}

\def\e{\varepsilon}

\def\vol{\mbox{\rm vol}}

\begin{document}
\hfill\today
\bigskip
\author[A. Fish, F. Nazarov, D. Ryabogin,  and A. Zvavitch]{Alexander Fish, Fedor Nazarov, Dmitry Ryabogin,  and Artem Zvavitch}
\title[Intersection body operator in the neighborhood of the ball]{ The behavior of iterations of the intersection body operator in a small neighborhood of the unit ball}

\address{Department of Mathematics,
University of Wisconsin, Madison
480 Lincoln Drive Madison, WI 53706 }\email{afish@math.wisc.edu}\address{Department of Mathematics,
University of Wisconsin, Madison
480 Lincoln Drive Madison, WI 53706 }\email{nazarov@math.wisc.edu}
\address{Department of Mathematics, Kent State University,
Kent, OH 44242, USA} \email{ryabogin@math.kent.edu}
\address{Department of Mathematics, Kent State University,
Kent, OH 44242, USA} \email{zvavitch@math.kent.edu}
\thanks{Supported in part by U.S.~National Science Foundation grants
 DMS-0652684, DMS-0800243, DMS-0808908.} 

\begin{abstract}
The intersection body of a ball is again a  ball. So, the unit ball $B_d \subset \R^d$
 is a fixed point of the intersection body operator acting on the space of all star-shaped origin symmetric bodies endowed with the Banach-Mazur distance.
 We show that this fixed point is a local attractor, i.e., that the iterations of the intersection body operator applied to any star-shaped origin symmetric body sufficiently
 close to $B_d$ in Banach-Mazur distance converge to $B_d$ in Banach-Mazur distance.
  In particular, it follows that the intersection body operator has no other fixed or periodic points in a small neighborhood of $B_d$.
\end{abstract}

\maketitle
\section{Introduction}

The notion of an {\it intersection body of a star body} was
 introduced by E. Lutwak \cite{Lu1}:
$K$ is called the intersection body of $L$ if the radial function of $K$ in
every direction is equal to the $(d-1)$-dimensional volume of the
central hyperplane section of $L$ perpendicular to this direction:
\begin{equation}
\label{eq:int}\rho_K(\xi)= \mbox{\rm vol}_{d-1} (L\cap \xi^\bot),
\mbox{ } \forall \xi\in S^{d-1},
\end{equation}
where $\rho_K(\xi)=\sup\{a: a\xi \in K\}$ is the radial function of
the body $K$ and $\xi^\bot=\{x \in \R^d:  (x, \xi)=0\}$ is the
central hyperplane perpendicular to the vector $\xi$.
 Using the formula for the volume in polar coordinates in $\xi^\bot$, we derive the following analytic definition of an {\it
intersection body of a star body}: $K$ is the intersection body of
$L$ if
$$
\rho_K(\xi)= \frac{1}{d-1}\RRR \rho^{d-1}_L (\xi):=\frac{1}{d-1}
\int\limits_{ S^{d-1}\cap \xi^{\perp}}\rho^{d-1}_L(\theta)d\theta.
$$
Here $\RRR$ stands for the spherical Radon transform.  We refer  the
reader to books \cite{Ga}, \cite{K} for more information on the
definition and properties of intersection bodies of star bodies and
their applications. 

Let us denote by $\I L$ the intersection body of a  body $L$. Let
$\S_d$ be the set of all star-shaped origin symmetric bodies in
$\R^d$ endowed with the Banach-Mazur distance
$$d_{BM}(K,L)=\inf\{b/a:\,\exists\,\, T \in GL(d) \mbox{  such that
 } aK\subseteq TL\subseteq b K\}.$$
We note that $\I(T L)= |\operatorname{det} T| (T^{*})^{-1}(\I L), \mbox{ for all  }
T\in GL(d) $ (see Theorem 8.1.6 in \cite{Ga}), hence the action of
$\I$ on $\S_d$ is well defined,  and $d_{BM}(\I(TK),\I(TL))=d_{BM}(\I K,\I L)$.

The action of $\I$ on $\S_2$ is quite simple; since $\I L$ is just $L$ rotated  by $\pi/2$ and stretched $2$ times, we have $\I L=L$ in $\S_2$, so every point of $\S_2$ is a fixed point of $\I$.

Let  $B_d$ be the unit Euclidean ball. We have
$$\rho_{\I(B_d)}(\xi)= \mbox{\rm vol}_{d-1}(B_d\cap \xi^\bot)=\mbox{\rm
vol}_{d-1}(B_{d-1}).
$$
Thus,  $ B_d$ is a fixed point of $\I$ in $\S_d$.

\noindent{\bf Question:} {\it  Do there exist any other fixed or
periodic points of $\I$ in $\S_d$, $d \ge 3$?}

In this paper we show that there are no such points in a small neighborhood of the ball $B_d$. This will immediately follow from the following

\begin{theorem}\label{t:main}
$$
\I^m L \stackrel{\S_d} {\longrightarrow}B_d \mbox{ as } m\to \infty,
$$
for all $L$ sufficiently close to $B_d$ in the Banach-Mazur distance.
\end{theorem}

  More
information on this and analogous questions can be found in Chapter
8 of \cite{Ga} (see Problems 8.6 and 8.7 page 337 and note
 8.6 on page 341) and \cite{Lu2}, \cite{GZ}.

We also note that a similar  question for projection bodies (see
\cite{Ga}, \cite{K}) is much better understood.  It is quite easy to
observe that the projection body of a cube is again (a dilation of) 
a cube.  W. Weil (see \cite{W}) described the polytopes
that are stable under the projection body  operation. Still the general
question of the description of all fixed points remains open.

\noindent{\bf Notation:} For a convex body $K \subset \R^d$,
consider the following two quantities:
$$
d_\infty(K)=\inf\{ \| 1-\rho_{TK}\|_\infty:  T\in GL(d)\},
$$
$$
d_2(K)=\inf\{ ||1-\rho_{TK}\|_2:  T \in GL(d) \}.
$$
Note that in the small neighborhood of $B_d$, the quantity
$d_{\infty}(K)$ is comparable with $ \log d_{BM}(K,B_d)$.

In this paper, we will denote by $|u|$ the Euclidean norm of a vector
$u\in \R^d$. We will denote by $C$, $c$ constants depending on $d$
(dimension) only, which may change from line to line.

\section{Plan of the proof of  the Theorem.}

To avoid writing irrelevant normalization constants in formulae,
from now on, we shall denote by $\RRR$ the normalized Radon
transform on $S^{d-1}$ that differs from the usual one by the factor
$\frac{1}{\vol_{d-2}(S^{d-2})}$, so $\RRR 1 =1$. It doesn't change anything because homotheties have already been factored out in the definition of $\S_d$.

Our main tool in the  proof  of Theorem \ref{t:main} is spherical
harmonics. We refer the reader to \cite{Gr} for more information and
definitions. We denote by $\H_k$ the space of spherical harmonics of
degree $k$. We shall denote by $H_k^f$ the projection of $f$ to
$\H_k$, so
$$
f \sim \sum_{k \ge 0} H_k^f.
$$
The following formula for  the Radon transform of a
spherical harmonic  $H_k \in \H_k$  of even order $k$ is especially
useful for our calculations (see Lemma 3.4.7 in \cite{Gr}):
\begin{equation}\label{eq:radon}
\RRR \,H_k= (-1)^{k/2}v_{d,k}H_k,
\end{equation}
where
$$
v_{d,k}=\frac{1\cdot 3\cdot \dots \cdot (k-1)}{(d-1)(d+1)\dots(d+k-3)} \approx k^{-(d-2)} .
$$

Let $K\in \S_d$ be close to $B_d$. Our main goal is to show the
following two things:
\begin{enumerate}
\item $\I^m K$ is smooth for all large $m$.
\item  If $K$ is sufficiently smooth and close to $B_d$, then $d_2(\I K) \le \lambda d_2(K)$ with some $\lambda <1$.
\end{enumerate}

The first claim will follow from the smoothing properties of $\RRR$.
 Since $f:S^{d-1}\to \R$ is $C^m$-smooth essentially if the norms of $H_k^f$
 decay as $k^{-m}$ and since $\RRR f \sim  \sum\limits_{k \ge 0} (-1)^{k/2} v_{d,k} H_k^f$, we conclude that the order of smoothness of $\RRR f$ exceeds the order
  of smoothness of $f$ by roughly speaking $d-2 \ge 1$.

Raising  $f$ to the power $d-1$ does not change its smoothness class
but can drastically increase the norm of $f$ in that class, so we
shall need some accurate computation to show that the smoothing
effect still prevails if $f$ is close to constant.

To prove the second claim, we write $\rho_K=1+\varphi$, where
$\varphi$  is an even function with small  $L^\infty$-norm and $\int_{S^{n-1}}\varphi =0$.
Then
$$
\rho_{\I K}= 1+(d-1) \RRR \varphi + \RRR \,O(\varphi^2).
$$
The main idea is to try to show that $\|(d-1) \RRR \varphi\|_{L^2}
\le \lambda \|\varphi\|_{L^2}$ with some $\lambda <1$. Since
$\|\varphi^2\|_{L^2}=O(\|\varphi\|_{L^\infty}\|\varphi\|_{L^2})$,
and $\|\RRR\|_{L^2\to L^2} \le 1$,  we get $\|\RRR O(\varphi^2)\|_{L^2} \le  C   \|\varphi\|_{L^\infty}\|\varphi\|_{L^2} $. Thus, 
$$
\|\RRR O(\varphi^2)\|_{L^2} \le\frac{1-\lambda}{2}\|\varphi\|_{L^2},
\mbox{  provided  that}  \|\varphi\|_{L^\infty} \le \frac{1-\lambda}{2},
$$
so the last term won't give us any trouble.

 Note that $\varphi
\sim \sum\limits_{l\ge 1} H_{2l}^{\varphi}$ and the terms
$H_{2l}^{\varphi}$ are orthogonal. If all the products
$v_{d,2l}(d-1)$ were less than $1$, our task would be trivial.
Unfortunately, $v_{d,2}(d-1)=1$ (but $v_{d,2l}(d-1)\le \frac{3}{d+1}
\le \frac{3}{4}$, for $l>1$). Thus, we need to kill $H_2^\varphi$
somehow. It turns out that it can be done by first applying a
suitable linear transformation to $K$.

\begin{remark} The proof below can be noticeably  shortened  in the case of  convex  bodies. Then we may use the
Busemann theorem (see \cite{Bu} or \cite{MP}, Theorem 3.9;
\cite{Ga}, Theorem 8.10) to claim that $\I^mL$ is convex, for all  $m\ge 1$, and  compare  $L^\infty$ and $L^2$ norms of radial functions of convex bodies directly, avoiding the smoothening procedure.
\end{remark}

\section{Auxiliary Lemmata.}

For a function $f : S^{d-1} \to \R$ we define its homogeneous extension $\check{f}$ of degree $0$ by
$$
\check{f}(x)=  f\left(\frac{x}{|x|}\right),
$$
so if $f$ is a smooth function on $S^{d-1}$, then $\check{f}$  is a smooth function on $\R^d\setminus  \{0\}$. By $Df$ and $D^2 f$, we mean the restrictions to the unit sphere $S^{d-1}$ of the first and the second differentials of $\check{f}$. Note that $D\check{f}$ and  $D^2\check{f}$ are homogeneous functions on $\R^d \setminus \{0\}$ of degree $-1$ and $-2$ respectively, so the norms $\|D f\|_{L^\infty}$ and $\|D^2 f\|_{L^\infty}$ do not bound the differentials $D \check{f} $ and $D^2 \check{f} $ on  the entire space $\R^d \setminus  \{0\}$. Still they bound them (up to a constant factor) outside any ball of positive radius centered at the origin, which is enough to transfer to the sphere all usual estimates coming from the second order Taylor formula in $\R^d$.

\begin{lemma}\label{metric}
Suppose that $f:S^{d-1} \to \R$ satisfies $\|D^2 f \|_{L^\infty} \le 1$, $\|f\|_{L^2}<\e$, for some $\e\in (0,1)$. Then
$\|f\|_{L^\infty} \le C \e^{\frac{4}{d+3}}$ and $\|D f\|_{L^\infty} \le C  \e^{\frac{2}{d+3}}$.
\end{lemma}
\begin{proof}
Replacing $f$ by $-f$, if necessary, we may assume that
$$\|f\|_{L^\infty}=\max\limits_{S^{d-1}} f = f(x_0)=M >0.$$
Since $D_{x_0}f=0$, we can use the second order Taylor formula to
conclude that
$$
f(x) \ge M - C \|D^2 f\|_{L^\infty} |x-x_0|^2 \ge M-C|x-x_0|^2.
$$
Thus, in the ball of radius $c\sqrt{M}$ (if $M$ is very large then this ball is just $S^{d-1}$), centered at $x_0$, we have
$$
f(x) \ge M-C c^2 M\ge \frac{1}{2}M, \mbox{  provided that } c^2C < \frac{1}{2}.
$$
Hence,
$$
\e^2 \ge \int\limits_{S^{d-1}} f^2 \ge c' \frac{M^2}{4} (\sqrt{M})^{d-1}=c'M^{\frac{d+3}{2}}
$$
if $c \sqrt{M} < 1$, or
$$
\e^2 \ge \frac{M^2}{4},
$$
if $c \sqrt{M} \ge 1$. In  both cases the first 
 inequality follows immediately.

The second inequality can now easily be derived from the classical Landau-ÐKolmogorov inequality (see \cite{HLP})
$$
\|Df\|_{L^\infty} \le C \|f\|^{\frac{1}{2}}_{L^\infty} \|D^2f\|^{\frac{1}{2}}_{L^\infty}.
$$
  \end{proof}

 Let $T \in GL(d)$.
 We would like to define the action of T  on bounded functions on $S^{d-1}$ in such a way that, for the radial function $\rho_K(x)=\|x\|_K^{-1}$ of a star-shaped  body $K$, the image $T\rho_K$ would coincide  with the radial function of $T^{-1}K$. To this end, note that 
 $$
\rho_{T^{-1}K}(x)= \|Tx\|^{-1}_K= \left\| \frac{Tx}{|Tx|} \right\|^{-1}_K |Tx|^{-1}=\rho_K\left(\frac{Tx}{|Tx|}\right)|Tx|^{-1}.
$$
Thus for an arbitrary bounded function  $f: S^{d-1} \to \R$, it is natural to put
 \begin{equation}
 Tf(x) := f\left(\omega_T(x)\right) | Tx|^{-1},
 \end{equation}
 where  $\omega_T: S^{d-1}\to S^{d-1}$ is given by $\omega_T(x) = \frac{Tx}{|Tx|}$.

\begin{lemma}\label{ssylka}
Let $T=I+Q$, where $Q$ is self-adjoint and $\Vert Q \Vert < \frac{1}{2}$.
Then
$$
|\omega_T(x)-x|\le C \Vert Q \Vert \mbox{ for all } x \in S^{d-1}.
$$
\end{lemma}
\begin{proof}
$$
\left|\omega_T(x)-x\right|=\frac{1}{|Tx|}\Big|Tx-|Tx|x\Big| \le
\|T^{-1}\| \Big|(Tx-x) - (|Tx|-1)x\Big|
$$
$$ \le
\|T^{-1}\|\Big[|Tx-x|+\left||Tx|-1\right| \Big] \le 2
\|T^{-1}\|\|Q\| \le \frac{2}{1-\|Q\|}\|Q\|.
$$
\end{proof}

\section{Classes $\U_{\alpha}$}

Let $ \alpha \ge 0$. For a bounded function $f$ on $S^{d-1}$, define
$\Vert f \Vert_{\U_{\alpha}}$ to be the least constant $M$ such that
$\Vert f \Vert_{L^{\infty}} \le M$ and for every $n \ge 1$, there
exists a polynomial $p_n$ of degree $n$ satisfying $\Vert f -
p_n\Vert_{L^2} \le M n^{-\alpha}$. We will say that $f\in
\U_{\alpha}$ if $\|f\|_{\U_{\alpha}}<\infty$.

Fix  an  infinitely smooth function $\Theta$  on $[0,+\infty)$ such
that $\Theta=1$ on $[0,1]$, $\Theta= 0$ on $[2,+\infty)$, and $0
\le \Theta\le1$ everywhere.

\noindent Consider the multiplier operator
\begin{equation}\label{mult}
\M_n f = \M_n^{\Theta} f = \sum_{k \ge 0}
\Theta\left(\tfrac{k}{n}\right) H_k^f.
\end{equation}
We will use the following property: $\Vert\M_n\Vert_{L^p \rightarrow L^p} \le
C(\Theta)$ for all $1\le p \le \infty$. This result is well known to experts but, for the sake of completeness, we will present a proof  in Appendix.

Note that $\M_n f$
is a polynomial of degree  $2n$. Also $\M_n p_n = p_n$ for all
polynomials $p_n$ of degree $n$.

Suppose now that $f \in \U_{\alpha}$. Let $q_n = \M_n f$.  We have
$$
\Vert f - q_n \Vert_{L^2} = \Vert (f-p_n) - \M_n (f-p_n) \Vert_{L^2}
\le C \Vert f - p_n \Vert_{L^2} \le C \Vert f \Vert_{\U_{\alpha}} n
^{-\alpha},
$$
and
$$
\Vert q_n \Vert_{L^{\infty}} \le C \Vert f \Vert_{L^{\infty}} \le C
\Vert f \Vert_{\U_{\alpha}}.
$$

Now we use the polynomials $q_n$ to prove the following  lemma describing the properties of the  classes
$\U_{\alpha}$.
\begin{lemma}
\label{lemma_sub_main} \phantom{vasya}\hfill

\begin{enumerate}
\item If $f,g \in \U_{\alpha}$, then $fg \in \U_{\alpha}$ and $\Vert
fg \Vert_{\U_{\alpha}} \le C \Vert f \Vert_{\U_{\alpha}} \Vert g
\Vert_{\U_{\alpha}}$.
\item Let $T\in GL(d)$  with $ \Vert T \Vert$,   $\Vert
T^{-1} \Vert \le 2$. Then, for every $\delta > 0$, $ f \in
\U_{\alpha}$, we have $Tf \in \U_{\alpha - \delta}$ and $\Vert T f
\Vert_{\U_{\alpha - \delta}} \le C_{\delta} \Vert f
\Vert_{\U_{\alpha}}$.
\item If $f \in \U_{\alpha}$, then $\RRR f \in \U_{\alpha+d-2}$ and
$\Vert \RRR f \Vert_{\U_{\alpha+d-2}} \le C \Vert f
\Vert_{\U_{\alpha}}$.
\end{enumerate}
\end{lemma}
\begin{proof}
\noindent (1) We obviously have
$$
\|fg\|_{L^\infty}\le\|f\|_{L^\infty}\|g\|_{L^\infty}\le\|f\|_{U_\alpha}\|g\|_{U_\alpha}.
$$
Now notice that
$$
\|f-\M_n f\|_{L^2}\le C \|f\|_{U_\alpha} n^{-\alpha} \mbox{ and }
\|g-\M_n g\|_{L^2}\le C \|g\|_{U_\alpha} n^{-\alpha},
$$
$$
\|\M_n f\|_{L^\infty}\le C \|f\|_{L^\infty}\le C \|f\|_{U_\alpha}
$$
$$
\|\M_n g\|_{L^\infty}\le C \|g\|_{L^\infty}\le C \|g\|_{U_\alpha},
$$
and that $p_n=\M_n f \cdot M_ng$ is a polynomial of degree $4n$. Hence
\begin{eqnarray*}
\|fg-p_n\|_{L^2}&=&\|(f-\M_nf)g+ M_n f(g-M_ng)\|_{L^2}\\
&\le&\|(f-\M_nf)\|_{L_2}\|g\|_{L^\infty}+ \|M_n
f\|_{L^\infty}\|(g-M_ng)\|_{L^2}\\&\le&
C\|f\|_{U_\alpha}\|g\|_{U_\alpha} n^{-\alpha}.
\end{eqnarray*}

\noindent(2) Write $f=p_n+g$ where $p_n=\M_nf$ and $\|g\|_{L^2}\le C 
\|f\|_{U_\alpha} n^{-\alpha}$. We have
$$
\left(Tf\right)(x)=|Tx|^{-1} f(\omega_T(x))=|Tx|^{-1}p_n(\omega_T(
x))+|Tx|^{-1}g(\omega_T(x)).
$$
Since $|Tx|^{-1} \le \|T^{-1}\| \le 2$ on $S^{d-1}$ and $\omega_T$
is a diffeomorphism of the unit sphere with bounded volume distortion coefficient, the $L^2$-norm of the second
term does not exceed $C\|g\|_{L^2} \le
C\|f\|_{U_\alpha}n^{-\alpha}$. Note now that $x\to |Tx|^{-1}$ is a
$C^\infty$-function and $\omega_T$ is a $C^\infty$-mapping on
$S^{d-1}$. Moreover, their derivatives of all orders are bounded by some constants depending on the dimension and the order, but not on $T$ (as long as $\|T\|$,  $\|T^{-1}\| \le 2$).

We need the following approximation lemma (see for example \cite{R}, Theorem 3.3):

\begin{lemma}   If
$m \in {\mathbb N}$, $h\in C^m(S^{d-1})$, then for every $N$, there
exists a polynomial $P_N$ of degree $N$ such that
$\|h-P_N\|_{L^2}\le C_m\|h\|_{C^m} N^{-m}$.
\end{lemma}

 Since both the multiplication by a $C^\infty$-function and a
 $C^\infty$ change of variable are bounded operators in $C^m$, the
 function $h(x)=|Tx|^{-1}p_n(\omega_T(x))$ belongs to $C^m$ and
 $\|h\|_{C^m}\le C_m\|p_n\|_{C^m}$. By the Bernstein inequality (see
 Theorem 3.2.6 in \cite{S}),
 $$
 \|p_n\|_{C^m}\le C_m\|p_n\|_{L^\infty}n^m\le C_m\|f\|_{L^\infty}n^m
 \le  C_m\|f\|_{U_\alpha}n^m.
 $$
Thus we can find a polynomial $P_N$ of degree $N=n^{1+\e}$ such that
$$
\|h-P_N\|_{L^2} \le C_m \|f\|_{U_\alpha} n^m N^{-m}= C_m
\|f\|_{U_\alpha}N^{-\frac{\e}{1+\e}m}.
$$
Consider some $\delta>0$ and choose $\e$ so small that $\frac{\alpha}{1+\e}
>\alpha-\delta$ and $m$ so large that $\frac{\e}{1+\e}m >
\alpha -\delta$.  Then we shall get
\begin{eqnarray*}
\|Tf-P_N\|_{L^2} &\le&  C_m
\left(N^{-(\alpha-\delta)}+n^{-\alpha}\right)\|f\|_{U_\alpha} \le C_m\left(N^{-\alpha-\delta}+N^{-\frac{\alpha}{1+\e}} \right) \|f\|_{U_\alpha} \\
 &\le& C_m N^{-(\alpha-\delta)}\|f\|_{U_\alpha}.
\end{eqnarray*}

\noindent (3) Obviously, $\| \RRR f \|_{L^{\infty}} \le \|f\|_{L^\infty} \le \|f\|_{U_\alpha}$.
Let $\Psi = 1 -\Theta$. Note that $\RRR\, \M_n f$ is a polynomial of
degree $2n$ and
\begin{eqnarray*}
\|\RRR f- \RRR \M_n f \|_{L^2}^2 &=& \sum\limits_{k\ge n} v^2_{d,k}
\Psi\left(\tfrac{k}{n}\right)^2 \|H^f_k\|^2_{L^2}\\
&\le& C n^{-2(d-2)}\sum\limits_{k\ge n}
\Psi\left(\tfrac{k}{n}\right)^2 \|H^f_k\|^2_{L^2}\\
&=& C n^{-2(d-2)}\|f-\M_n f\|^2_{L^2}\\
&\le& C \|f\|^2_{U_\alpha} n^{-2(d-2+\alpha)}.
\end{eqnarray*}

\end{proof}

\begin{lemma}
\label{l:haha} Let $\beta>\alpha$. Then for every $\sigma>0$, there
exists $C=C_{\sigma, \alpha, \beta}>0$ such that $\Vert f \Vert_{\U_{\alpha}}\le
C \Vert f \Vert_{L^{\infty}}+\sigma \Vert f
\Vert_{\U_{\beta}}$.
\end{lemma}
\begin{proof}
We have $\Vert f \Vert_{L^{\infty}}\le C \Vert f
\Vert_{L^{\infty}}$ as soon as $C \ge 1$. Now take $n\ge 1$.
If $n^{-(\beta-\alpha)}>\sigma$, take $p_n=0$. Then, $$\Vert f-p_n
\Vert_{L^2}\le \Vert f \Vert_{L^{\infty}}\le C \Vert f
\Vert_{L^{\infty}}n^{-\alpha}, $$ provided that
$C >\sigma^{-\frac{\alpha}{\beta-\alpha}}$. If
$n^{-(\beta-\alpha)}\le\sigma$, choose $p_n$ so that
$$
\Vert f-p_n \Vert_{L^2}\le \Vert f
\Vert_{\U_{\beta}}n^{-\beta}=n^{-(\beta-\alpha)}\Vert f
\Vert_{\U_{\beta}}n^{-\alpha}\le \sigma\Vert f
\Vert_{\U_{\beta}}n^{-\alpha}.
$$

\end{proof}

\section{Iteration Lemma}

\begin{lemma}
\label{Main_Lemma}
 Fix $\alpha$ so large that $\U_{\alpha} \subset
C^2$. Let $L>0$ be a constant such that $ \Vert \cdot \Vert_{C^2}
\le L \Vert \cdot \Vert_{\U_{\alpha}}$. There exist $\eps_d > 0$
and $\lambda_d <1$ with the following property. For every $\eps \in (0,\eps_d)$ and
every function $f$ such that $f=1 + \varphi$, $\int \varphi = 0$,
$\Vert \varphi \Vert_{L^2} \le \eps, \Vert \varphi
\Vert_{\U_{\alpha}} \le L^{-1}$, there exists a linear operator $ T
\in GL(d)$ and a positive number $\gamma$ such that $\widetilde{f} =
\gamma \RRR (Tf)^{d-1}$ can be written as $1+ \widetilde{\varphi}$ where
$\int \widetilde{\varphi} = 0$, $\Vert \widetilde{\varphi} \Vert_{L^2} \le
\lambda_d \eps$, $\Vert \widetilde{\varphi}\Vert_{\U_{\alpha}} \le
L^{-1}$.
 \end{lemma}
\begin{proof}

\noindent{\bf Step 1:}  We show first that
 there exists an operator  $T$, such that  $Tf=1+\psi$, where
  $\|\psi\|_2 \le \e +C
\e^{\frac{d+5}{d+3}}$  and  $\|H_2^{\psi}\|_2 \le C
\e^{\frac{d+5}{d+3}}$.

We shall seek $T$ in the form $T=I+Q$  as in Lemma \ref{ssylka}. We
have
$$
|Tx|=\sqrt{1+2(Qx,x)+\Vert Q\Vert^2}=1+(Qx,x)+O(\Vert Q\Vert^2).
$$
Hence,
$$
|Tx|^{-1}=1-(Qx,x)+O(\Vert Q \Vert^2).
$$
Further, since $\|\varphi\|_{C^2} \le L ||\varphi\|_{U_\alpha} \le 1$,
Lemmata \ref{metric}, \ref{ssylka} yield
$$
|\varphi(\omega_T(x))-\varphi(x)|\le C
\eps^{\frac{2}{d+3}}|\omega_T(x)-x|\le C\eps^{\frac{2}{d+3}}\Vert
Q\Vert.
$$
We also have
\begin{eqnarray}
Tf(x)&=&|Tx|^{-1}(1+\varphi(\omega_T(x)))\nonumber\\&=&(1-(Qx,x)+O(\Vert
Q\Vert^2))(1+\varphi(x)+O(\eps^{\frac{2}{d+3}}\Vert Q\Vert))\label{eq:ssu}\\
&=&1-(Qx,x)+\varphi(x)+O(\Vert Q\Vert\eps^{\frac{2}{d+3}}+\Vert
Q\Vert^2).\nonumber
\end{eqnarray}
Now we choose $Q$ so that $(Qx,x)=H^{\varphi}_2(x)$. Since $\Vert
H^{\varphi}_2 \Vert_{L^2}\le \Vert \varphi \Vert_{L^2}\le\eps$, and
$H^{\varphi}_2$ is a quadratic polynomial, we can conclude that all
its coefficients do not exceed $C\eps$ and thereby $\Vert
Q\Vert=O(\eps)$. Also, applying Lemma \ref{metric} we get   $\|\varphi\|_{L^\infty} \le C
\e^{\frac{4}{d+3}}$. Thus, by (\ref{eq:ssu}), $Tf=1+\psi$, where
$\psi=\varphi-H^{\varphi}_2+O\left(\eps^{\frac{d+5}{d+3}}\right)$.
Note now that
$$
\|\varphi-H^{\varphi}_2\|_{L^2}\le \|\varphi\|_{L^2}\le \eps,
$$
so $\|\psi\|_{L^2}\le \eps+O(\eps^{\frac{d+5}{d+3}})$, and  that
$\varphi-H^{\varphi}_2$ has no spherical harmonics of degree $2$ in
its decomposition, so $\Vert
H^{\psi}_2\Vert_{L^2}=O\left(\eps^{\frac{d+5}{d+3}}\right)$. Also 
\begin{equation}\label{eq:psi}
\|\psi\|_{L^\infty}\le C\e^{\frac{4}{d+3}}.
\end{equation}

\noindent{\bf Step 2:} Now we compute $(Tf)^{d-1}$. We have
$$
(Tf)^{d-1}=(1+\psi)^{d-1}=1+(d-1)\psi+\eta,
$$
and (\ref{eq:psi}) yields
$$
\|\eta\|_{L^2}\le C\eps^{\frac{4}{d+3}}\|\psi\|_{L^2}\le
C\eps^{\frac{d+7}{d+3}}.
$$
Applying the Radon transform, we get
$$
\RRR (Tf)^{d-1}=1+(d-1) H^{\psi}_0+H^{\eta}_0+ (d-1)\RRR H^{\psi}_2
+(d-1)\RRR(\psi-H_0^{\psi}-H^{\psi}_2)+\RRR(\eta-H^{\eta}_0).
$$
Note that $(d-1)H^{\psi}_0+H^{\eta}_0$ is a constant function whose value $\zeta$  satisfies $|\zeta| \le \|\psi\|_{L^2} \le C\eps$. We also have
$$
(d-1)\|\RRR H^{\psi}_2\|_{L^2}=\|H^{\psi}_2\|_{L^2}\le
C\eps^{\frac{d+5}{d+3}},
$$
$$
(d-1)\|\RRR(\psi-H^{\psi}_0- H^{\psi}_2)\|_{L^2}\le
\lambda_d\,\|\psi\|_{L^2},
$$
and $$\|\RRR(\eta-H^{\eta}_0)\|_{L^2}\le \|\eta\|_{L^2}\le
C\eps^{\frac{d+7}{d+3}}.
$$
Now take $\gamma=(1+\zeta)^{-1}=1+O(\eps)$ and put
$$
\widetilde{\varphi}=\gamma(\RRR
H^{\psi}_2+(d-1)\RRR(\psi-H^{\psi}_0-
H^{\psi}_2)+\RRR(\eta-H^{\eta}_0)).
$$
Note that
$$
\|\widetilde{\varphi}\|_{L^2}\le
(1+O(\eps))(\lambda_d\eps+O(\eps^{\frac{d+5}{d+3}}))=\lambda_d\eps+O(\eps^{\frac{d+5}{d+3}})<\lambda_d'
  \eps,
$$
with any $\lambda_d < \lambda_d' < 1$ provided that $\eps$ is small enough. Also
$\int\widetilde{\varphi}=0$, and
$\gamma\RRR(Tf)^{d-1}=1+\widetilde{\varphi}$. At last
$$
\|\widetilde{\varphi}\|_{L^{\infty}}\le
C\left(\|\psi\|_{L^{\infty}}+\|\eta\|_{L^{\infty}}\right)\le
C\eps^{\frac{4}{d+3}}.
$$

\noindent{\bf Step 3:}  It remains to estimate
$\|\widetilde{\varphi}\|_{\U_{\alpha}}$. Note that
$\|f\|_{\U_{\alpha}}\le 2$, so applying Lemma \ref{lemma_sub_main}, with $\delta=1/2$, we get
 $$\|Tf\|_{\U_{\alpha-\frac{1}{2}}}\le C \,\Rightarrow\, \|(Tf)^{d-1}\|_{\U_{\alpha-\frac{1}{2}}}\le
C'\,\Rightarrow\,\|\RRR(Tf)^{d-1}\|_{\U_{\beta}}\le
C''\,\Rightarrow\,\|\widetilde{\varphi}\|_{\U_{\beta}}\le C''',
$$
where $\beta=\alpha-\frac{1}{2}+d-2>\alpha$. Now choose $\sigma>0$ so that
$C'''\sigma\le \frac{1}{2L}$. Then, by Lemma \ref{l:haha},
$$
\|\widetilde{\varphi}\|_{\U_{\alpha}}\le \sigma
C'''+C_{\sigma,\alpha,\beta}C'\eps^{\frac{4}{d+3}}\le \frac{1}{L},
$$
provided that $\eps$ is small enough.

\end{proof}
\section{Smoothing}

Fix $\beta > \alpha>0$. Let $f=1+\varphi$, $ \Vert
\varphi\Vert_{L^{\infty}}<\eps<1/2$. Define the sequence $f_k$
recursively by $f_0=f$, $f_{k+1}=\RRR f_k^{d-1}$. Using Lemma \ref{lemma_sub_main},
we can conclude that $f_k\in\U_{\beta}$ for sufficiently large $k$
and $\Vert f_k\Vert_{\U_{\beta}}\le C(k)$. Also, it is easy to show
by induction that $$ (1-\eps)^{(d-1)^k}\le f_k\le
(1+\eps)^{(d-1)^k}.
$$
 Let $\mu=\int f_k$. If $\eps>0$ is sufficiently small, then
 $|\mu-1|$ is small and $\mu^{-1} f_k=1+\psi$ where $\int \psi=0$ and $\Vert
\psi\Vert_{L^{\infty}}$ is small. Note that
$$
\Vert \psi \Vert_{\U_{\beta}}\le 1+\mu^{-1}\Vert
f_k\Vert_{\U_{\beta}}\le C'(k),
$$
and, thereby, by Lemma \ref{l:haha}, $\Vert \psi \Vert_{\U_{\alpha}}$ is also
small ($\Vert \psi \Vert_{\U_{\beta}}$ is bounded by a fixed
constant and $\Vert \psi \Vert_{L^{\infty}}\to 0$ as $\eps\to 0$).
Applying this observation to the function $\rho_K$, we conclude that
if $K$ is sufficiently close to $B_d$, then,  after
proper normalization,   $\rho_{\I^kK}$  can be written as $1+\varphi$ with $\Vert \varphi
\Vert_{\U_{\alpha}}$ as small as we want.

\section{The end of the proof}
 Now we choose $\e$ so small that the
smoothing part results in a body $K$ for which $\rho_K$ satisfies
the assumptions of Lemma \ref{Main_Lemma}. Then $\rho_{K_1}$, where $K_1=\gamma\I
T K$ satisfies the assumptions of Lemma \ref{Main_Lemma}  with $\lambda \e$
instead of $\e$. Note that $K_1\stackrel{\S_d}{=}\I K$. Applying
Lemma \ref{Main_Lemma}  again, we get a body $K_2\stackrel{\S_d}{=}\I^2 K$ such
that $\rho_{K_2}$ satisfies the assumption of Lemma \ref{Main_Lemma}  with
$\lambda^2\e$ instead of $\e$ and so on.

In particular, it means that
$$
\|\rho_{K_m}-1\|_{L^2} \le \lambda^m\e \to 0  \mbox{ as } m\to\infty
$$
and $\|\rho_{K_m}\|_{C^2} \le 2$.

This is enough to conclude that
$$
d_{BM}(K_m, B_d)=d_{BM}(\I^m K, B_d) \to 0 \mbox{  as  } m\to\infty.
$$

\section{Appendix}
\begin{prop}\label{th:bound}
Consider $\Theta\in C^{\infty}_0(\R)$.  Then the operator  $\M_n^\Theta$ defined in (\ref{mult}) is bounded in $L^p$, for all $1\le p \le \infty$, i.e.
\begin{equation}\label{eq:1}
\|\M_n^\Theta f\|_{L^p(S^{d-1})}\le C\,\|f\|_{L^p(S^{d-1})}.
\end{equation}
\end{prop}

The proposition is well known to the specialists but to make the paper self-contained, we present its proof below.

We start the proof  with some auxiliary lemmata.  We assume below that the measure $\sigma$ on the sphere is normalized so that the total measure of the sphere is one.

For every  $z\in \C$ such that $|z|<1$, define the  function $P_z(\mathbf{x},\mathbf{y}): S^{d-1}\times S^{d-1} \to \C$ by
\begin{equation}
P_z(\mathbf{x},\mathbf{y}):=\frac{1-z^2}{(1+z^2-2z(\mathbf{x}\cdot \mathbf{y}))^{d/2}},\qquad z\in
\C,\,\,|z|<1,
\end{equation}
where for odd $d$ we pick the branch of an analytic function $$z\to
g(z)=(1+z^2-2z(\mathbf{x}\cdot \mathbf{y}))^{d/2}$$ in such a way that $g(\R_+)\subset
\R_+$.
\begin{lemma} \label{ zto|z|}   For all $x,y\in S^{d-1}$, and $|z|<1$
$$
|P_z(\mathbf{x},\mathbf{y})|\le 2 \cdot
3^d\Big(\frac{|1-z|}{1-|z|}\Big)^{d+1}P_{|z|}(\mathbf{x},\mathbf{y}).
$$
\end{lemma}
\begin{proof}  For $\beta\in \C$, $|\beta|=1$, we have
$$
\frac{||z|-\beta|}{|z-\beta|}\le 1+\frac{|z-|z||}{|z-\beta|}\le 1+
\frac{|z-|z||}{||z|-1|}\le \frac{|z-|z||+||z|-1|}{1-|z|}\le
3\,\frac{|1-z|}{1-|z|}.
$$
We also have
$$
\frac{|1-z^2|}{1-|z|^2}\le 2\,\frac{|1-z|}{1-|z|}.
$$
Since
$$
1+z^2-2z(\mathbf{x}\cdot \mathbf{y})=(z-\alpha)(z-\bar{\alpha}), \mbox{ for } \alpha=\mathbf{x}\cdot
\mathbf{y}+ i\sqrt{1-(\mathbf{x}\cdot \mathbf{y})^2},
$$
we conclude
$$
\frac{|P_z(\mathbf{x},\mathbf{y})|}{P_{|z|}(\mathbf{x},\mathbf{y})}=\frac{|1-z^2|\,\,|(|z|-\alpha)(|z|-\bar{\alpha})|^{d/2}}{|1-|z|^2|\,\,|(z-\alpha)(z-\bar{\alpha})|^{d/2}}\le
2 \cdot  3^d\Big(\frac{|1-z|}{1-|z|}\Big)^{d+1}.
$$
\end{proof}

\begin{lemma} Let $z\in \C$, $0< Im \,z <2$, and let $n\in
\N$. Then,
$$
\|P_{e^{iz/n}}(\mathbf{x},\cdot)\|_{L^1(S^{d-1})}\le 2^{d+2} \cdot 3^d\,\Big(\frac{|z|}{Im \,z}\Big)^{d+1}.
$$
\end{lemma}
\begin{proof}
Put $\xi=iz/n$. Then,
\begin{eqnarray*}
\frac{|1-e^{\xi}|}{1-e^{\xi}}&\le &1+
\frac{|e^\xi-e^{Re\,\xi}|}{1-e^{Re\,\xi}}\le
1+\frac{e^{Re\,\xi}|Im\,\xi|}{1-e^{Re\,\xi}}=1+\frac{|Im\,\xi|}{e^{-Re\,\xi}-1}\le
1+\frac{|Im\,\xi|}{|Re\,\xi|}\\
&\le&
\frac{2|\xi|}{|Re\,\xi|}=2\frac{|z|}{Im \,z}.
\end{eqnarray*}
Now by  Lemma   \ref{ zto|z|},
$$
|P_{e^{iz/n}}(\mathbf{x},\mathbf{y})|\le 2 \cdot
3^d\Big(\frac{|1-e^{iz/n}|}{1-|e^{iz/n}|}\Big)^{d+1}P_{|e^{iz/n}|}(\mathbf{x},\mathbf{y})\le
2^{d+2} \cdot 3^d \Big(\frac{|z|}{{\rm Im} \,z}\Big)^{d+1} P_{|e^{iz/n}|}(\mathbf{x},\mathbf{y}).
$$
It remains to use $\|P_{|e^{iz/n}|}(\mathbf{x},\cdot)\|_{L^1(S^{d-1})}=1$.
\end{proof}

Let $S(\R)$ be the Schwartz space.  To prove (\ref{eq:1}),  write
\begin{equation}\label{eq:2}
\Theta\left(\frac{k}{n}\right)=\int\limits_{\R}\psi(x)e^{ikx/n}dx,
\end{equation}
where $\psi \in S(\R)$ is the Fourier  transform of some $C_0^\infty$ extension of $\Theta$ to the entire real line.

Using the Stokes formula, we can rewrite the last integral as
$$
2i\int\limits_{{\rm Im} \,  z> 0 }\bar{\partial}\Psi(z)e^{ikz/n}dA(z),
$$
where $\Psi$ is any reasonable extension of $\psi$ to the upper half-plane.  To make this representation useful, we shall need the following lemma:

\begin{lemma}\label{bound} For any $\psi\in S(\R)$ there exists an
extension $\Psi(z)$, ${\rm Im}\,z\ge 0$, $\Psi |_{\R}(x)=\psi(x)$, such
that
$$
 \int\limits_{{\rm Im} z>0}\Big|\,\bar{\partial}\Phi(z)\Big|\Big(\frac{|z|}{{\rm Im}\,z}\Big)^{d+1}dA(z)<\infty.
$$
\end{lemma}

Let us first show that Lemma \ref{bound}  gives $\|M_n^\Theta \|_{L_p \to L_p} < \infty$. Indeed,  using
(\ref{eq:2}), we can calculate the kernel $K_n$ of the operator $M_n^\Theta$,
$$
\M_n^\Theta f=\sum\limits_{k=0}^{\infty}\Theta(k/n)H_k^f
=\int\limits_{\R}\psi(x)\sum\limits_{k=0}^{\infty}e^{ikx/n}H_k^f dx=
$$
$$
2i\int\limits_{ {\rm Im} \, z>
0 }\bar{\partial}\Psi(z)\sum\limits_{k=0}^{\infty}e^{ikz/n}H_k^f dA(z).
$$
Now note that
$$
\sum\limits_{k=0}^{\infty}e^{i k z/n}H_k^f(\mathbf{x})=\int\limits_{S^{d-1}}P_{e^{iz/n}}(\mathbf{x},\mathbf{y})f(\mathbf{y})d\sigma(\mathbf{y}).
$$
So,
$$
K_n=
2i\int\limits_{{\rm Im} \, z>
0 }\bar{\partial}\Psi(z)P_{e^{iz/n}} dA(z).
$$
Since
$\|K_n(\mathbf{x},\cdot)\|_{L^1(S^{d-1})}\le C$,  we have $\|K_n(\cdot,\mathbf{y})\|_{L^1(S^{d-1})}\le C$ by symmetry.
Now (\ref{eq:1}) follows from the Schur test.

Let us now prove  Lemma \ref{bound} :

\begin{proof} We define
$$
\Psi(x+iy)=\eta(y)\Psi_0(x+iy),\qquad
\Psi_0(x+iy)=\sum\limits_{k=0}^{d+1}\psi^{(k)}(x)(iy)^k/k!,
$$
where $\eta:[0,\infty)\to [0,1]$ is infinitely differentiable,
$\eta(y)=1$ for $0\le y\le 1$, and $\eta(y)=0$ for $y\ge 2$. Observe
that
$$
|2\bar{\partial}\Psi(x+iy)|=|(\partial /\partial x+i\partial/\partial
y)\Psi(x+iy)|\le $$
$$ |\Psi_0(x+iy)||\eta'(y)|+
|\eta(y)|\Big|\psi^{(d+2)}(x)(iy)^{d+1}/(d+1)!\Big|.
$$
Hence,
$$
\int\limits_{{\rm Im} \,z>0 }\Big|\,\bar{\partial}\Phi(z)\Big|\Big(\frac{|z|}{Im\,z}\Big)^{d+1}dA(z)\le
2\int\limits_{{\rm Im}\,z\le
2 }|\Psi_0(x+iy)||z|^{d+1}dA(z)+
$$
$$
\frac{1}{(d+1)! }\int\limits_{{\rm  Im} \,z\le
2 }|\psi^{(d+2)}(x)||z|^{d+1}dA(z)\le C,
$$
and we are done, since $\psi\in S(\R)$.\end{proof}


\begin{thebibliography}{99}


\bibitem[Bu]{Bu}{\sc H. Busemann},  {\it A theorem of convex bodies of the Brunn-Minkowski type,}
Proc. Nat. Acad. Sci. U.S.A.  35, (1949), 27-31.




 \bibitem[HLP]{HLP}{\sc G.H. Hardy,  E. Littlewood,  G. Polya,} {\it  Inequalities,}
Cambridge University Press, 1934.

\bibitem[Ga]{Ga} {\sc R.J.~Gardner}, {\it Geometric tomography},
Cambridge Univ. Press, New York, 1995.


\bibitem[GZ]{GZ} {\sc E. Grinberg and Gaoyong Zhang}, {\it Convolutions, Transforms, and Convex Bodies},
Proc. London Math. Soc. (3)  78 (1999), 77-115.

\bibitem[Gr]{Gr} {\sc H. Groemer}, {\it Geometric Applications of Fourier Series and Spherical Harmonics},  Cambridge University Press, New York, 1996.

\bibitem[K]{K} {\sc A.~Koldobsky},
{\it Fourier Analysis in Convex Geometry}, Math. Surveys and Monographs,
AMS, Providence RI 2005.



\bibitem[Lu1]{Lu1} {\sc E.~Lutwak}, {\it Intersection bodies and dual mixed volumes},
Advances in Math. 71 (1988), 232--261.

\bibitem[Lu2]{Lu2} {\sc E.~Lutwak}, {\it Selected affine isoperimetric inequalities}, In a Handbook of Convex Geometry, ed. by P.M. Grubet and J.M. Wills. North-Holland, Amsterdam, 1993, pp. 151-176.
Advances in Math. 71 (1988), 232--261.
\bibitem[MP]{MP}{\sc V.D. Milman, A. Pajor, } {\it Isotropic position and
inertia ellipsoids and zonoids of the unit ball of a normed
$n$-dimensional
 space},  Geometric aspects of functional analysis (1987--88),
Lecture Notes in Math., 1376, Springer, Berlin, (1989), 64-104.

\bibitem[R]{R}{\sc D.L. Ragozin,}  {\it Constructive polynomial approximation on spheres and projective spaces}.  Trans. Amer. Math. Soc.  162  1971 157--170.

\bibitem[S]{S}{\sc H.S. Shapiro,}  {\it Topics in Approximation theory,} Lecture Notes in Math., 1376, Springer, Berlin, 1971.

\bibitem[W]{W}{\sc W. Weil}, {\it {\" U}ber die Projektionenk{\" o}rper konvexer Polytope.}  Arch. Math. (Basel) 22 (1971), 664--672.

\end{thebibliography}
 \end{document}